 \documentclass[11pt,letterpaper]{amsart}

 \usepackage{txfonts} 

 \usepackage[colorlinks,linkcolor=blue,citecolor=blue]{hyperref}
 
 \theoremstyle{plain}

\newtheorem{theorem}{Theorem}[section]
\newtheorem{lemma}[theorem]{Lemma}

\theoremstyle{definition}

\newtheorem{corollary}[theorem]{Corollary}

  \newcommand\dd{\,\mathrm{d}}

 \newcommand{\Ric}{\mbox{Ric}}

\theoremstyle{remark}
\newtheorem{remark}[theorem]{Remark}

\numberwithin{equation}{section}

\begin{document}

\title[Negatively curved  metrics] 
 {The existence of negatively curved  metrics on locally conformally flat manifolds with boundary}


 \author{Rirong Yuan }
 \address{School of Mathematics, South China University of Technology, Guangzhou 510641, China}
 \email{yuanrr@scut.edu.cn}



\dedicatory{}

\begin{abstract}
	
We use certain Morse functions to construct conformal metrics with negative sectional curvature on locally conformally flat manifolds with boundary. 
	Moreover, without conformally flatness assumption, we also construct conformal metric of positive Einstein tensor.
\end{abstract}

\maketitle


  
  \section{Introduction}

  Let $(\bar M^n,g)$ be a connected compact  Riemannian manifold of dimension $n\geq 3$  with non-empty smooth  boundary $\partial M$,  $\bar M=M\cup\partial M$. Here  $M$ denotes the interior  of $\bar M$.
 Let  $K_g$, ${Ric}_g$ and   ${R}_g$ denote the sectional,  Ricci and scalar curvature of  $g$, respectively, with respect to the  Levi-Civita connection $\nabla$. Denote the Einstein tensor by $$G_g=Ric_g-\frac{R_g}{2}\cdot g.$$   
 The Riemannian curvature tensor $\mathrm{Riem}_g$ 
  can be decomposed as follows  (cf. \cite{Besse1987}):
 \begin{equation}
 	\label{decomposion1-0}
 	\begin{aligned}
 		\mathrm{Riem}_g= {W}_g+ A_g \odot g,  
 	\end{aligned}
 \end{equation}  
 where $W_g$ is the Weyl curvature tensor of $g$, 
 and $\odot$ stands for the Kulkarni-Nomizu product, as well as $A_g$ denotes the Schouten tensor 
 \[A_g=\frac{1}{n-2}\Big(Ric_g-\frac{R_g}{2(n-1)} g\Big).\]
  
  In Riemannian geometry, 
  a basic problem is to find a metric so that the various curvatures satisfy prescribed properties. A well-known result 
   on this direction is 
  due to 
  Gao-Yau \cite{Gao1986Yau} who proved that any closed 3-manifold admits a Riemannian metric with negative Ricci curvature. Subsequently, Gao-Yau's theorem was  extended
   by Lohkamp
  \cite{Lohkamp-1,Lohkamp-2} to higher dimensional manifolds possibly with boundary. 
Sectional curvature seemingly makes the problem harder.
On every  3-manifold with   boundary,
 Hass \cite{Hass1994} constructed   a metric
 such that with respect to the metric the manifold has negative sectional	curvature and the boundary is   concave outwards.  More recently, based on a Morse theory technique that was introduced by  \cite{yuan-PUE1}, the author  \cite{yuan-PUE3-construction} 
  constructed  negatively curved metrics in each conformal class of 3-manifolds with boundary. That is


 	 		\begin{theorem}
 	 			[{\cite[Theorem 1.1]{yuan-PUE3-construction}}]
 	 			\label{thm0-main-dimension3}
 	 			Let $(\bar M,g)$ be a 3-dimensional compact connected  Riemannian manifold with smooth boundary. There is a smooth compact  conformal 
 	 			metric $g_u=e^{2u}g$ of negative sectional curvature.
 	 			
 	 		\end{theorem}

Note  in dimension three that the metric with negative sectional curvature must has  positive Einstein tensor  and vice versa.
This paper is devoted to exploring higher-dimensional analogues by generalizing Theorem \ref{thm0-main-dimension3} from the perspective of negative sectional curvature and positive Einstein tensor.

We first consider the case of positive Einstein tensor.
\begin{theorem}
	\label{thm2-Einstein1} 
	Each compact connected Riemannian manifold 
	with smooth boundary admits a smooth compact conformal metric 
	of positive Einstein tensor. 
\end{theorem}

Next, we turn to the case of negative sectional curvature.
From the 
curvature tensor  decomposition \eqref{decomposion1-0}, the
  locally conformally flat metric is of special interest, in which case the
Weyl tensor   vanishes,  and then 
	\begin{equation}
	\label{decomposion1}
	\begin{aligned}
		\mathrm{Riem}_g=A_g \odot g.
	\end{aligned}
\end{equation}
 Consequently, on a locally conformally flat manifold, 
one can potentially gather information about the Riemannian curvature tensor effectively through the Schouten tensor, which is a symmetric (0,2)-tensor and is much simpler to handle.
  
 Motivated by the decomposition \eqref{decomposion1}, 
 in this paper we prove that every locally conformally flat manifold with boundary is pointwisely conformal to a negatively curved manifold.
 
 \begin{theorem}
 	\label{thm0-main-LCF}
 	Let  $(\bar M^n,g)$ be  a compact connected  locally conformally flat Riemannian manifold  with smooth boundary. There exists a smooth compact  conformal 
 	metric $g_u=e^{2u}g$ with negative sectional curvature.
 \end{theorem}

\begin{remark}	
	 Note that the decomposition \eqref{decomposion1} holds on every  3-manifold. 
	 In dimension three,    Theorems \ref{thm2-Einstein1} and \ref{thm0-main-LCF} reduce to Theorem \ref{thm0-main-dimension3}, which is itself a special case of  \cite[Theorem 1.2]{yuan-PUE3-construction}.
\end{remark}

  The proof is based on the 	 Morse theory argument that was proposed  by \cite{yuan-PUE1} and subsequently developed in \cite{yuan-PUE2-note,yuan-PUE3-construction}. 
In our proof of Theorems \ref{thm2-Einstein1} and \ref{thm0-main-LCF}, we follow the treatment in \cite[Theorem 1.2]{yuan-PUE3-construction} closely.

 \vspace{1mm}
The article is organized as follows.  
In Section \ref{sec2-cons} we construct admissible functions. 
As applications, we give the proof of  in Section \ref{sec-constructionofmetric}.
 
 \vspace{1mm}
 
  The author  would like to thank  Professor Yi Liu for generously 
 answering questions related to the proof of Lemma \ref{lemma-diff-topologuy}.
  The author also    wishes to   thank Professor Jiaping Wang for drawing his attention to 
  \cite{Hass1994}.
 The author is  supported by    Guangdong Basic and Applied Basic Research Foundation (Grant No. 2023A1515012121), and  
 Guangzhou Science and Technology Program  (Grant No. 202201010451).



\section{Construction of admissible functions}
\label{sec2-cons}

 \subsection{Some result on Morse function}

The following lemma asserts that every compact connected manifold with boundary carries a Morse function without any critical point.    

\begin{lemma}
	\label{lemma-diff-topologuy}
	Let $\bar M$
	be a compact connected  
	manifold of dimension $n\geq 2$ with smooth boundary. Then there is a smooth function $v$ without any critical points. 
\end{lemma}

\begin{proof} 
	Let $X$ be the double of $M$. Let $w$ be a smooth Morse function on $X$ with the critical set $\{p_i\}_{i=1}^{m+k}$, among which $p_1,\cdots, p_m$ are all the critical points  being in $\bar M$. 
	Pick $q_1, \cdots, q_m\in X\setminus \bar M$ but not the critical point of $w$. By homogeneity lemma 
	(see \cite{Milnor-1997}), 
	one can find a diffeomorphism
	$h: X\to X$, which is smoothly isotopic to the identity, such that 
	\begin{itemize}
		\item $h(p_i)=q_i$, $1\leq  i\leq  m$.
		\item $h(p_i)=p_i$, $m+1\leq  i\leq  m+k$.
	\end{itemize}
	Then $v=w\circ h^{-1}\big|_{\bar M}$ is the desired 
	function. That is, $d v\neq 0$.
	
\end{proof}


 \subsection{Construction of admissible functions} 
We  improve some of results in  \cite{yuan-PUE1,yuan-PUE2-note,yuan-PUE3-construction}.
Let  
$\Gamma\subsetneq \mathbb{R}^n$ be an \textit{open},  \textit{symmetric}, \textit{convex} cone in $\mathbb{R}^n$ with  vertex at the origin, 
non-empty boundary $\partial \Gamma$, 
and
$$\Gamma_n:=\left\{\lambda=(\lambda_1,\cdots,\lambda_n)\in \mathbb{R}^n: \mbox{ each } \lambda_i>0\right\}\subseteq\Gamma.$$
Denote $\vec{\bf1}=(1,\cdots,1)\in\mathbb{R}^n$, $(1,\cdots,1, 1-\varrho_\Gamma)\in\partial\Gamma$ and $\bar \Gamma=\Gamma\cup\partial\Gamma $.

Let $(\bar M^n,g)$ be a connected compact  Riemannian manifold of dimension $n\geq 3$  with smooth boundary. For a symmetric $(0,2)$-tensor $A$, let $\lambda(g^{-1}A)$ denote an $n$-tuple of of the eigenvalues of $A$ with respect to $g$.
 Let $\mathrm{U}(x)$ be a smooth symmetric $(0,2)$ tensor on $\bar M$.
Let    
$\alpha $, $\beta $ and 
$\varrho $ be given constants. 
Consider 
\begin{equation}
	\label{def1-V}
	\begin{aligned}
		V[u]=\nabla^2u+\alpha |\nabla u|^2 g-\beta  du\otimes du+ R(x,\nabla u)+\mathrm{U}(x)
	\end{aligned}
\end{equation}	
and 
\begin{equation}
	\label{def2-W}
	\begin{aligned}
		W[u]=\Delta u\cdot  g-\varrho \nabla^2u+\alpha |\nabla u|^2 g-\beta  du\otimes du+ R(x,\nabla u)+\mathrm{U}(x)
	\end{aligned}
\end{equation}	
where   
$R(x,p)$ is symmetric $(0,2)$ tensor smoothly depends on $T\bar M$.
We are interesting in the following two cases:
\begin{enumerate}

	\item 	There exists a positive continues function $\gamma(x,p)$ 
	with $\lim_{|p|\to+\infty}\gamma(x,p)=0$ uniformly, such that 
	\begin{equation}
		\label{assump2-R}
		\begin{aligned}
			|R(x,p)|\leq \gamma(p)(1+|p|^2), \,\, \forall (x,p)\in T\bar M. 
		\end{aligned}
	\end{equation}
	
		\item There is a uniform constant $C$ such that 
	\begin{equation}
		\label{assump1-R}
		\begin{aligned}
			|R(x,p)|\leq C(1+|p|), \,\, \forall (x,p)\in T\bar M. 
		\end{aligned}
	\end{equation}

\end{enumerate}
\begin{remark}

The considered cases include the Schouten tensor, and modified Schouten tensor 
	\begin{equation}
		\begin{aligned}
			A_{{g}}^{\tau,\zeta}=\frac{\zeta}{n-2} \left({Ric}_{g}-\frac{\tau}{2(n-1)}   {R}_{g}\cdot {g}\right), \,\, \alpha=\pm1,  \,\, \tau \in \mathbb{R}, \nonumber
		\end{aligned}
	\end{equation}
	as well as the Bakry-Emery  curvature tensor and more general $N$-Ricci curvature $$	\Ric_{N,\mu}=\Ric + \nabla^2 V-\frac{\mathrm{d} V\otimes \mathrm{d} V}{N-n}$$ on a metric measure space $(M,\mathrm{d},\mu)$,  
	where $\mathrm{d} \mu=e^{-V}\mathrm{d} \mathrm{vol}$.  
	
	Under the conformal deformation $g_u=e^{2u}g$, 
	\begin{equation}
		\label{conformal-Schouten1}
		\begin{aligned}
			-A_{g_u}
			= \,& 
			- A_{g} +\nabla^2 u	+\frac{1}{2}|\nabla u|^2  \cdot g- du\otimes du,  
		\end{aligned}
	\end{equation}
	\begin{equation}
	 \label{conformal-formula1}
		\begin{aligned}
			A_{{g}_u}^{\tau,\zeta}
			= A_{g}^{\tau,\zeta}
			+\frac{\zeta(\tau-1)}{n-2}\Delta u \cdot g-\zeta  \nabla^2 u
			+\frac{\zeta(\tau-2)}{2}|\nabla u|^2\cdot g
			+\zeta  \mathrm{d}{u}\otimes\mathrm{d}{u}, 
		\end{aligned}
	\end{equation}
	and 
	\begin{equation}
		\label{formula1-generalized-Ricc}
		\begin{aligned}  
			-\Ric_{N,\mu}(g_u)= \,&
			\Delta u\cdot g 	+(n-2)\nabla^2 u	
			+(n-2)(|\nabla u|^2 \cdot g-\mathrm{d} u\otimes \mathrm{d} u)
			\\\,&
			+\mathrm{d} u\otimes \mathrm{d} V+\mathrm{d} V\otimes \mathrm{d} u 
			-\langle\nabla u,\nabla V \rangle_g \cdot g 
			-\Ric_{N,\mu}(g).  
		\end{aligned}
	\end{equation}

\end{remark}

Using the Morse theory technique, we prove 
\begin{theorem}
	\label{main-result1}
	Suppose   one of the following holds
	\begin{itemize}	
		\item[$\mathrm{(i)}$]   
		$		(\alpha,\cdots,\alpha,\alpha-\beta) \in  \Gamma$, and  	$R(x,p)$ satisfies \eqref{assump2-R}. 
		
		\item[$\mathrm{(ii)}$]  Suppose $\alpha>0$,	 $(\alpha,\cdots,\alpha,\alpha-\beta) \in\partial \Gamma$, 
		and  	$R(x,p)$ satisfies \eqref{assump1-R}. 
		
	\end{itemize} 
	Then there is a   function $u\in C^\infty(\bar M)$ such that 
	\begin{equation}
		\label{admissible-2}
		\begin{aligned} 
			\lambda(g^{-1}V[u])\in\Gamma \textrm{ in } \bar M.  \nonumber 
		\end{aligned}
	\end{equation}
	
\end{theorem}

\begin{remark}
	 	The positivity of $\alpha$ in \eqref{def1-V} plays crucial roles in the case $\mathrm{(ii)}$.  
	 In general one could not expect a similar construction for $\alpha\leq0$.
	 A specific example is the obstruction to the existence of metric of positive Schouten tensor   in each conformal class.
	 Note  that the  locally conformally flat metric of positive Schouten tensor has positive sectional curvature.
\end{remark}

\begin{proof}
By Lemma \ref{lemma-diff-topologuy}, we can pick a    smooth function  $v$ with $\mathrm{d}v\neq 0$ and $v\geq 1$ in $\bar M$.
As in 
\cite{yuan-PUE1}, also in \cite{yuan-PUE2-note,yuan-PUE3-construction}, take
\begin{equation}
	\label{admissible-metric2}
	\begin{aligned}
		{u}=e^{Nv}.    
	\end{aligned}
\end{equation}
	By straightforward computation
	\begin{equation}
		\begin{aligned}
			V[u]
			= \,&
			N^2e^{2Nv} (\alpha |\nabla v|^2 \cdot g-\beta \mathrm{d}v\otimes \mathrm{d}v)  +N^2e^{Nv}\mathrm{d}v\otimes \mathrm{d}v
			\\ \,&
			+R(x,Ne^{Nv}\nabla v)
			+Ne^{Nv}\nabla^2 v	+U.  \nonumber
		\end{aligned}
	\end{equation}

Notice that $\mathrm{d}v\neq 0$ in $\bar M$.
Under the assumption \eqref{assump2-R}, for any $\epsilon>0$ 
one can find $N_\epsilon>0$ (depending on $\epsilon^{-1}$) such that for any $N\geq N_\epsilon$ 
\begin{equation}
	\label{assump2-R-2}
	\begin{aligned}
		|R(x,Ne^{Nv}\nabla v)|\leq \epsilon  N^2 e^{2Nv}|\nabla v|^2.  
	\end{aligned}
\end{equation} 
Similarly, if $R(x,p)$ satisfies \eqref{assump1-R} then one can find a uniform constant $C$ such that 
\begin{equation}
	\label{assump1-R-2}
	\begin{aligned}
		|R(x,Ne^{Nv}\nabla v)|\leq C  N  e^{Nv}|\nabla v|, \textrm{ for } N\geq 1.  
	\end{aligned}
\end{equation}  

	{\bf Case $\mathrm{(i)}$}: Fix a $\epsilon_0>0$ so that $(\alpha,\cdots,\alpha,\alpha-\beta)\in 2\epsilon_0\vec{\bf1}+\Gamma$.
	Take $\epsilon=\epsilon_0$ in \eqref{assump2-R-2}, then
	\begin{equation}
		\begin{aligned}
			\lambda\big(g^{-1}[N^2e^{2Nv} (\alpha |\nabla v|^2 \cdot g-\beta \mathrm{d}v\otimes \mathrm{d}v)
			+R(x,Ne^{Nv}\nabla v) ] \big) \in \epsilon_0 N^2e^{2Nv}|\nabla v|^2 \vec{\bf1}+\Gamma.  \nonumber
		\end{aligned}
	\end{equation}
Obviously 
$|Ne^{Nv}\nabla^2 v+U|\leq C(1+Ne^{Nv}).$
So  
	$\lambda(g^{-1}V[u])\in\Gamma$ if $N\gg1$.
	
	{\bf Case $\mathrm{(ii)}$}: 
	We follow closely the treatment in the proof of \cite[Theorem 1.2]{yuan-PUE3-construction}  to produce $ \frac{\alpha }{\beta} N^2 e^{Nv}|\nabla v|^2 \cdot g$.
	More precisely, 
	we  rewrite $V[u]$ as follows
	\begin{equation}
		\begin{aligned}
			V[u]
			= \,&
			\frac{1}{\beta} N^2e^{Nv}(\beta e^{Nv}-1) (\alpha |\nabla v|^2 \cdot g-\beta \mathrm{d}v\otimes \mathrm{d}v)  
			\\\,&
			+ \frac{\alpha }{\beta} N^2 e^{Nv}|\nabla v|^2 \cdot g 
			+R(x,Ne^{Nv}\nabla v)
			+Ne^{Nv}\nabla^2 v	+U.  \nonumber
		\end{aligned}
	\end{equation}
Note that   $ \beta=  \varrho_\Gamma \alpha$ in this case. We have $N_1$ such that for any $N\geq N_1$ 
	\begin{equation}
		\begin{aligned}
			\beta e^{Nv}-1 =\varrho_\Gamma \alpha e^{Nv}-1\geq 0,   \nonumber
		\end{aligned}
	\end{equation}   
	\begin{equation}
		\begin{aligned}
			\frac{\alpha }{\beta} N^2 e^{Nv}|\nabla v|^2 \cdot g 
			+R(x,Ne^{Nv}\nabla v)
			+Ne^{Nv}\nabla^2 v	+U \geq 	\frac{1}{2\varrho_\Gamma}  N^2 e^{Nv}|\nabla v|^2 \cdot g.  \nonumber
		\end{aligned}
	\end{equation}
In the last inequality, we also used \eqref{assump1-R-2}.
	From this,   $\lambda(g^{-1}V[u])\in\Gamma$ if $N\gg1$.
	
\end{proof}


\begin{theorem}
	\label{main-result2} 
	Suppose   one of the following holds
	\begin{itemize}	
		\item[$\mathrm{(i)'}$]   
		$(\alpha,\cdots,\alpha,\alpha-\beta) \in  \Gamma$, and  	$R(x,p)$ satisfies \eqref{assump2-R}. 
		
		\item[$\mathrm{(ii)'}$]  $\varrho<\varrho_\Gamma$,  $(\alpha ,\cdots,\alpha,\alpha -\beta ) \in\partial\Gamma$, and  	$R(x,p)$ satisfies \eqref{assump1-R}. 
		
		\item[$\mathrm{(iii)'}$]    $\varrho=\varrho_\Gamma$, $\alpha \varrho_\Gamma-\beta>0$,	$(\alpha ,\cdots,\alpha,\alpha -\beta ) \in\partial\Gamma$, and  	$R(x,p)$ satisfies \eqref{assump1-R}.

	\end{itemize} 
	Then there is a   function $u\in C^\infty(\bar M)$ such that 
	\begin{equation}
		\label{admissible-3}
		\begin{aligned} 
			\lambda(g^{-1}W[u])\in\Gamma \textrm{ in } \bar M.  \nonumber
		\end{aligned}
	\end{equation}
	
\end{theorem}

\begin{proof}
	
	Let  ${u}=e^{Nv}$ be as in \eqref{admissible-metric2}.
	The straightforward computation gives  
	\begin{equation}
		\begin{aligned}
			W[u]=\,& 
			N^2 e^{Nv} (|\nabla v|^2 \cdot g -\varrho \mathrm{d}v\otimes \mathrm{d}v)
			+N^2e^{2Nv} (\alpha |\nabla v|^2 \cdot g -\beta \dd v\otimes \dd v)
			\\\,&
			+Ne^{Nv} (\Delta v \cdot g-\varrho\nabla^2v)	+R(x,Ne^{Nv}\nabla v)+U. \nonumber
		\end{aligned}
	\end{equation}
	
In the cases $(i)'$ and  $(ii)'$, we have a positive  constant $\epsilon_0$ such that $(\alpha,\cdots,\alpha,\alpha-\beta) \in    2\epsilon_0\vec{\bf1}+\Gamma$ and $(1,\cdots,1,1-\varrho)\in 2\epsilon_0\vec{\bf1}+\Gamma$, respectively. 
Accordingly, for $N\gg1$ we have  $\lambda(g^{-1}W[u])\in\Gamma$ under the assumptions \eqref{assump2-R} and \eqref{assump1-R}.


The rest is to deal with the case $(iii)'$. 
First we prove  
\begin{equation}
	\label{beta-negative1}
	\beta<0.
\end{equation}
 Suppose by contradiction that   $\beta\geq 0$. Then $\alpha\varrho_\Gamma>\beta\geq 0$, and thus $\alpha>0$. Combining with $(\alpha ,\cdots,\alpha,\alpha -\beta ) \in\partial\Gamma$, we derive $\beta=\alpha\varrho_\Gamma$, which contradicts to $\alpha\varrho_\Gamma>\beta$. 
Thus we obtain \eqref{beta-negative1}.
 Combining with $\alpha \varrho_\Gamma-\beta>0$, we see $1-\frac{\alpha \varrho_\Gamma}{\beta}>0$.
 
Similar to   the treatment in the proof of \cite[Theorem 1.2]{yuan-PUE3-construction}, we get 
	\begin{equation}
	\begin{aligned}
		W[u]=\,&  
	 N^2e^{Nv}\big(e^{Nv}+\frac{\varrho_\Gamma}{\beta}\big) (\alpha |\nabla v|^2 \cdot g -\beta \dd v\otimes \dd v)
	+(1-\frac{\alpha \varrho_\Gamma}{\beta})	N^2 e^{Nv} |\nabla v|^2 \cdot g 
	\\\,& 
		+Ne^{Nv} (\Delta v \cdot g-\varrho_\Gamma\nabla^2v)	
	  +R(x,Ne^{Nv}\nabla v)+U. \nonumber
	\end{aligned}
\end{equation} 

Under the assumption  \eqref{assump1-R}, we get \eqref{assump1-R-2}. Then
one can pick $N\gg1$ such that 
\begin{align*}
	\frac{1}{2}(1-\frac{\alpha \varrho_\Gamma}{\beta})	N^2 e^{Nv} |\nabla v|^2 \cdot g  
	+Ne^{Nv} (\Delta v \cdot g-\varrho_\Gamma\nabla^2v)	
+R(x,Ne^{Nv}\nabla v)+U\geq 0.
\end{align*}
Thus
\begin{align*}
		W[u] \geq  
	N^2e^{Nv}\big(e^{Nv}+\frac{\varrho_\Gamma}{\beta}\big) (\alpha |\nabla v|^2 \cdot g -\beta \dd v\otimes \dd v) 
	+\frac{1}{2}(1-\frac{\alpha \varrho_\Gamma}{\beta})	N^2 e^{Nv} |\nabla v|^2 \cdot g. 
\end{align*}
Thus
$\lambda(g^{-1}W[u])\in\Gamma$ if  $N\gg1$.

\end{proof}

 \section{Proof of main results}
  \label{sec-constructionofmetric}
  
 \begin{proof}
 	[Proof of Theorem \ref{thm2-Einstein1}]

 The Einstein tensor coincides with the modified Schouten tensor \eqref{conformal-formula1} with $\tau=n-1$ and $\zeta=1$.
 In this case,  $\varrho=1$,   $\alpha=\frac{n-3}{2}$ and $-\beta=1$; see \eqref{conformal-Schouten1}. 
As in \cite{yuan-PUE1}, take ${u}=e^{Nv}$, where  $v$ is a smooth function with $\mathrm{d}v\neq0$ and $v\geq1$ in $\bar M$. By Theorem \ref{main-result2}, we have a smooth conformal metric $g_u=e^{2u}g$ such that  $Ric_{g_u}-\frac{1}{2}R_{g_u}\cdot g_u>0$ in $\bar M$.
  \end{proof}

\begin{proof}
	[Proof of Theorem   \ref{thm0-main-LCF}]

 Denote $g_u=e^{2u}g$ is the desired metric. 
 Fix $x\in\bar M$. 
 Let $e_1,\cdots, e_n$ be an orthonormal basis of $T_x\bar M$ (with respect to the resulting $g_u$), and we may further assume the matrix
 $\{A_{g_u}(e_i,e_j)\}$ is diagonal at $x$.
 For $i\neq j$, let $\Sigma_{i,j}$  denote the tangent $2$-plane spanned by $e_i$ and $e_j$.
 By the decomposition \eqref{decomposion1} on the locally conformally flat manifold $(\bar M,g_u)$, we infer that 
 \begin{equation}
 	\label{formu-113}
 	\begin{aligned} 
 		K_{g_u} (\Sigma_{i,j})=A_{g_u} (e_i, e_i)
 		+A_{g_u} (e_j, e_j), \,\, \forall i\neq j.  \nonumber
 	\end{aligned}
 \end{equation} 
 From this, in order to complete the proof, it suffices to  find a  $u\in C^\infty(\bar M)$ such that
\begin{equation}
	\label{main-verification1}
	\begin{aligned}
		\lambda(-g^{-1}A_{g_u})\in\mathcal{P}_2.
	\end{aligned}
\end{equation}

To achieve this, as in Theorem \ref{main-result1} we   take ${u}=e^{Nv}$, where  $v$ is a smooth function with $\mathrm{d}v\neq0$ and $v\geq1$ in $\bar M$.
For the (minus) Schouten tensor $-A_{g_u}$, we see $\alpha=\frac{1}{2}$ and $\beta=1$; see \eqref{conformal-formula1}.
According to Theorem \ref{main-result1}, when $N\gg1$,   ${g}_u=e^{2{u}}g$ is the  desired conformal metric satisfying \eqref{main-verification1}.

\end{proof}
\bigskip


\end{document}